\documentclass{article}

\usepackage{latexsym}
\usepackage[utf8]{inputenc}
\usepackage{amsmath, amsfonts, amssymb, amsthm}
\usepackage{mathrsfs}
\usepackage{graphicx}
\usepackage{color}

\newtheorem{thm}{Theorem}
\newtheorem{defn}{Definition}
\newtheorem{lemma}{Lemma}
\newtheorem{prop}{Proposition}

\newtheorem{cor}{Corollary}

\newcommand{\C}{\mathbb C}
\newcommand*{\CC}{\widehat\C}
\newcommand{\Exc}{\mathcal{E}}
\newcommand*{\R}{\mathbb R}
\newcommand{\D}{\mathbb D}
\newcommand{\N}{\mathbb N}
\newcommand{\Z}{\mathbb Z}
\newcommand{\weil}{\mathcal{D}}
\newcommand{\wto}{\stackrel{\mathrm{w}^*}{\to}}
\newcommand{\cont}{\mathscr{C}}

\newcommand{\absv}[1]{\left\lvert{#1}\right\rvert}
\newcommand{\Supp}{{\operatorname{supp}}}%\ensuremath
\newcommand{\itemref}[1]{\ref{#1}.}
\newcommand{\thmref}[1]{Theorem~\ref{#1}}

\begin{document}
\title{Roots of polynomial sequences in root-sparse
regions.}

\author{%
  Henriksen, Christian \\
  \texttt{chrh@dtu.dk} \\
  Department of Applied Mathematics and Computer Science \\
  Technical University of Denmark
  \and
  Petersen, Carsten Lunde \\
  \texttt{lunde@math.ku.dk} \\
  Department of Mathematical Sciences \\
  University of Copenhagen
  \and
  Uhre, Eva \\
  \texttt{euhre@ruc.dk} \\
  Department of Science and Environment \\
  Roskilde University
}

\date{\today}

\maketitle

\begin{abstract}
Given a family $(q_k)_k$ of polynomials, we call an open set $U$ 
\emph{root-sparse} if the number of zeros of $q_k$ is locally uniformly
bounded on $U$.
We study the interplay between the individual zeros of the polynomials $q_k$ 
and those of the $m$th derivatives $q_k^{(m)}$, 
in a root-sparse open set $U$, as $k\to\infty$.
More precisely, if the root distributions $\mu_k$ of $q_k$ 
converge weak* to some compactly supported measure $\mu$, 
whose potential is nowhere locally constant 
on a root-sparse open set $U$, then we link the roots of the $m$th derivative $q_k^{m}$, for an arbitrary $m>0$, to the roots of $q_k$ 
and the critical points of the potential $p_\mu$ on compact subsets of $U$.

We apply this result in a polynomial dynamics setting to obtain convergence results for the roots of the $m$th derivative 
of iterates of a polynomial outside the filled-in Julia set.
We also apply our result in the setting of extremal polynomials.
\end{abstract}

% ***********************************************
%
\section{Introduction and main results} \label{intro}
%
% ***********************************************

Studies of the statistical 
properties of the root loci of general families $(q_k)_k$ 
of polynomials with degrees increasing to infinity,
have a long history.
Recently, the relation between the roots of $q_k$
and $q_k'$ has gathered a lot of interest,
see e.g.\ \cite{TotikBLMS2016}, \cite{TotikTAMS2019}, \cite{HPU1}, \cite{HPU2}, \cite{Okuyama1}
and \cite{OkuyamaVigny}.
Each $q_k$ has a root distribution $\mu_k$.
The root distributions $(\mu_k)_k$ 
form a pre-compact family of probability measures on the Riemann sphere. 
Hence, passing to a subsequence if necessary, the
sequence converges to some accumulation point $\mu$ 
of the full sequence $(\mu_k)_k$.

Given a family $(q_k)_k$ of polynomials we call an open set $U$ 
\emph{root-sparse} if the number of zeros of $q_k$ is locally uniformly
bounded on $U$ (see Definition~\ref{def:sparse}).
In this paper, we study the interplay between 
the individual zeros of the polynomials $q_k$ 
and those of the $m$th derivatives $q_k^{(m)}$, 
in a root-sparse open set $U$, as $k\to\infty$.

More precisely suppose the root distributions $\mu_k$ of $q_k$ 
converge weak* to some compactly supported measure $\mu$, 
whose potential is nowhere locally constant 
on a root-sparse open set $U$.
On compact subsets of such a root-sparse open set,
we can, for any $m$, link the roots of the $m$th derivative $q_k^{m}$ to the roots of $q_k$ and the critical points of the potential $p_\mu$ of $\mu$.
For precise statements see \thmref{main}
and \thmref{maintechnical} below.

For a typical example of root-sparsity, consider a Borel probability measure $\mu$ with non-polar compact support $K$.
Then the unbounded connected component of $\C\setminus K$ is root-sparse for the family of orthogonal polynomials defined by $\mu$.

Another range of examples is provided in polynomial
iteration.
Indeed, take any polynomial $P$ of degree at least $2$.
Let $K$ be the filled-in Julia set of $P$, and $U = \C\setminus K$.
Then $U$ is root-sparse for the sequence of iterates $q_k=P^{k}$.

In order to state a precise result, 
we need to specify what we understand by 
the root distribution of a polynomial,
root-sparsity
and matching of roots in root-sparse sets. 
For the latter, we use a notion of divisors 
and their convergence. 

Throughout the paper, $(q_k)_k = (q_k)_{k\in\N}$ 
denotes a sequence of polynomials of degrees $n_k>0$ tending to infinity.
As we are only interested in the location and multiplicity of the zeros, 
we can restrict our attention to monic polynomials
$q_k(z) = z^{n_k} + O(z^{n_k-1}) = \prod_{j=1}^{n_k} (z-z_{k,j})$
without loss of generality.
Here $z_{k,j}$, $j=1,\ldots, n_k$ are the roots of the polynomial $q_k$, 
repeated with multiplicity.

The \emph{root-counting measure} of $q_k$ is the measure $\mu_k^\#$ given by
\[
  \mu_k^\# = \sum_{j=1}^{n_k} \delta_{z_{k, j}},
\]
where $\delta_z$ denotes the Dirac point mass at $z$. 
The \emph{root distribution} $\mu_k$ of $q_k$ is the Borel probability measure obtained from normalizing
the root-counting measure
\[
  \mu_k 
  = \frac{1}{n_k} \mu_k^{\#}
  = \frac{1}{n_k}\sum_{j=1}^{n_k} \delta_{z_{k, j}}.
\]
With the aid of $\mu_k^\#$, we can formally define
root-sparsity.
\begin{defn}\label{def:sparse}
    An open set $U\subset\C$ is called \emph{root-sparse} (for $(q_k)_k$) 
    if for any compact set $K\subset U$, there exists $M = M(K)$ 
    such that 
    $$
    \mu_k^\#(K) \leq M,
    $$    
    for all sufficiently large $k$. 
\end{defn}
Clearly, an open root-sparse set $U$ does not meet the support of any accumulation point $\mu$ of $(\mu_k)_k$.

Root-sparse open sets show up in many interesting families of polynomials, e.g.\ sequences of orthogonal, Chebyshev, and Fekete polynomials,
for more examples see the applications below.

Let $U \subset \C$ be an open set.
We call a mapping $\xi : U \to \Z$ a \emph{divisor} on $U$ if the
set $\xi^{-1}(\Z\setminus\{0\})$ has no accumulation points in $U$.

We denote the $\Z$-module of divisors on $U$ by $\weil(U)$.
We consider the divisors on $U$ as elements of the dual space
of the vector space $\cont_c(U)$ of continuous functions on $U$ 
with compact support in $U$, $\weil(U)\subset\cont_c^*(U)$. 
That is, $\xi \in \weil(U)$ acts on $f \in \cont_c(U)$ by 
\[
  \xi(f) = \sum_z \xi(z) f(z),
\]
where the sum is taken over the finite set
$\{ z \in \Supp(f) : \xi(z) \neq 0\}$. 
This allows us to endow $\weil(U)$ with the weak* topology.
For a sequence $(\xi_k)_{k\in\N} \subset \weil(U)$,
and $\xi \in \weil(U)$,
we have $\xi_k \to \xi$ if and only if $\xi_k(f) \to \xi(f)$ 
for every $f \in \cont_c(U)$. 
Notice that if $\xi$ is nonnegative, then $\xi$ corresponds
to a measure, and our notion of convergence corresponds to vague
convergence of measures.

The sequence $(\xi_k)_{k\in\N} \subset \weil(U)$ is called \emph{locally bounded} if and only if for every $f\in\cont_c(U)$ the sequence $|\xi_k|(f)$ is bounded.

Given a non-constant meromorphic function $f:U\to \CC$, we can
associate a divisor $\xi_f\in\weil(U)$ to $f$ by:
\[
\xi_f(z) = 
\begin{cases}
    0,&\qquad \textrm{if $f(z)\notin\{0, \infty\}$,}\\
    m&\qquad \textrm{if $f$ has a zero of order $m$ at $z$,}\\
    -m&\qquad \textrm{if $f$ has a pole of order $m$ at $z$.}
\end{cases}
\]
Thus $\xi_f$ is nonnegative, whenever $f$ is holomorphic.

Given a harmonic function $h : U \to \R$,
we define $h' : U \to \R$ by 
\begin{equation}\label{eq:harmonicprime}
h'(z)
:= 2\frac{\partial}{\partial z} h(z)
= (\frac{\partial}{\partial x}
  - i \frac{\partial}{\partial y}) h(x+iy),
\end{equation}
where we have written $z = x + iy$ in the usual manner. 
It follows from the Cauchy-Riemann equations that 
$h'$ is holomorphic, and we can thus define the nonnegative divisor
$\xi_{h'}$ as above.

Our main result links, for each order of derivative $m\geq 1$, the three divisors in $\weil(U)$, the divisor $\xi_{k,m} := \xi_{q_k^{(m)}}$ of $q_k^{(m)}$, the divisor $\xi_k := \xi_{q_k}$ of $q_k$ and the divisor $\xi_{p_\mu'}$, in the case where the root distributions $\mu_k$ of $q_k$ converge to a limiting probability measure $\mu$ with potential $p_\mu$, which is non-constant on every connected component of $U$. 

\begin{thm}\label{main}
Suppose that $(\mu_k)_k$ 
converges weak* to a compactly supported measure $\mu$.
Let $U\subset \C$ be a root-sparse open set 
on which $p_\mu$ is nowhere locally constant.
Then 
  \[
    \xi_{k,m} - \xi_k 
    \underset{k\to\infty}{\longrightarrow}\ m \xi_{p_\mu'}
    \text{ on } U,\text{ for } m = 0, 1, 2, \ldots
  \]
\end{thm}
\noindent Note that the sequences of divisors above are locally uniformly bounded and it is understood that the convergence is weak* with respect to $\cont_c(U)$ as defined above.

The following corollary states,
that when $U$ has convex complement,
we do not need convergence of the root distributions,
as long as the root-loci are uniformly bounded.

\begin{cor}\label{corollaryconvex}
  Suppose that there exists $R > 0$,
  such that $\mu_k(\D(R))=1$, for all k,
  and suppose that $U\subset \C$ is a root-sparse open set such that $V = \C \setminus U$  is convex.
  Then
  \[
    \xi_{k,m} - \xi_k 
    \to 0 \text{ on $U$ for } m = 0, 1, 2, \ldots
  \]
\end{cor}

\begin{cor}\label{corollaryPonvex}
    Suppose that $(\mu_k)_k$ converges weak* to the equilibrium measure $\omega$ of some non-polar, polynomially convex, connected compact set $K$. 
    Suppose moreover that $U = \C\setminus K$ is root-sparse. 
    Then 
   \[
    \xi_{k,m} - \xi_k 
    \to 0 \text{ on $U$ for } m = 0, 1, 2, \ldots
  \]
\end{cor}

The proof of Theorem~\ref{main} relies on a more general result. 
The assumption that the root distributions $\mu_k$ converge can be relaxed to 
the assumption that their potentials $p_{\mu_k}$, properly adjusted by additive constants, converge to a limiting harmonic function $p$ on $U$, in the sense given in the following theorem.

\begin{thm}\label{maintechnical}
Let $U \subset \C$ be an open set
and $p : U \to \R$ a nowhere locally constant harmonic function
such that the following two conditions are satisfied.
\begin{enumerate}
    \item\label{uniformboundroots} 
    $U$ is root-sparse,
    \item\label{balayage} 
    for any infinite set $N_1\subset\N$ there exist an infinite subset
    $N_2\subset N_1$, a subset $\Exc\subset U$
    without accumulation points in $U$
    and a sequence $(d_k)_k\subset \R$ such that 
    \[
    d_k+p_{\mu_k}\to p,\quad
    \textrm{locally~uniformly in $U\setminus\Exc$ as $k\to\infty$ in $N_2$}.
    \]
\end{enumerate}
Then 
  \[
    \xi_{k,m} - \xi_k 
    \underset{k\to\infty}{\longrightarrow}\ m \xi_{p'} 
    \text{ on } U,\text{ for } m = 0, 1, 2, \ldots
  \]
\end{thm}

Before proving our results (in Section~\ref{sec:proofs}),
we apply them in a polynomial dynamics setting as well as in an
extremal polynomial setting.

\section{Prerequisites in potential theory}

Suppose $\mu$ is a finite Borel measure with compact non-polar support $K$. Its potential $p_\mu$ is defined as
\[
p_\mu(z)=\int\log|z-w|d\mu(w) = \mu(\C)\log\lvert z \rvert + o(1),
\]
where we follow the sign convention of \cite{Ransford}. 

The potential is harmonic on $V := \C \setminus K$.
The derivative $p_\mu'$ (defined by \eqref{eq:harmonicprime})
is thus well-defined and holomorphic on $V$,
where it is given by the
Cauchy transform of $\mu$
\begin{equation}\label{eq:cauchytransform}
p_\mu'(z)=\int\frac{1}{z-w} d\mu(w).
\end{equation}
It follows that if $\gamma$ is a curve in $V$ connecting $z_0$ to $z_1$, then 
\begin{equation}\label{eq:recuperatepot}
p_\mu(z_1)= p_\mu(z_0) + \Re\left(\int_\gamma p_\mu' \right).
\end{equation}

Let $\Omega$ denote the unbounded, connected component 
of the complement of $K$, and let $\omega$ 
denote the equilibrium probability measure for $K$. 
The Green's function $g_\Omega$ with pole at infinity satisfies $g_\Omega=p_\omega-I(\omega)$, where 
$$
I(\omega) = \iint\log|z-w|d\mu(w)d\mu(z)
$$ is the energy of $\omega$. 
The logarithmic capacity of a compact set $K$ is $c(K)=e^{I(\omega)}$.

% ***********************************************
%
\section{Applications}
%
% ***********************************************

Our first application comes from polynomial iteration.
See e.g. \cite{Milnor} for a general introduction.
Let $P : \C \to \C$ be a monic polynomial of degree $d > 1$.
We obtain a family $(q_k)_k$ of polynomials, by letting
$q_k = P^k$ denote the $k$th iterate of $P$.

Given $z \in \C$ there are two possibilities.
Either $q_k(z) \to \infty$ or $q_k(z)$, $k = 1, 2, \ldots$, form a bounded sequence.
The set of $z$ such that the former occurs is called the filled-in Julia set $K(P)$,
that is,
\[
  K(P) = \{z\in \C: (P^k(z))_k \text{ is a bounded sequence}\}.
\]
It is well known and easy to see that $K(P)$ is compact, nonempty,
and polynomially convex.
In particular, $\Omega(P) :=\C \setminus K(P)$ is connected.

Let $g_{\Omega}$ denote the Green's function associated to $\Omega(P)$.
For monic polynomials, it is equal to the potential $p_\omega$ associated to
the equilibrium distribution $\omega$ on $K$.
We obtain the following result from Theorem~\ref{maintechnical}.

\begin{prop}\label{pro:dyn}
  Let $C \subset \Omega(P)$ be compact and such that $\partial C$
  does not contain any critical points of $g_{\Omega}$.
  Let $s$ denote the number of critical points of $g_{\Omega}$
  in $C$ counted with multiplicity.
  Then, for any $m \in \{0, 1, \ldots\}$ there exists $k_0$,
  so that when $k \geq k_0$, the $m$th derivative $q_k^{(m)}$
  has exactly $sm$ roots in $C$ counted with multiplicity.
\end{prop}
\begin{proof}
We show that the conditions in Theorem~\ref{maintechnical} are satisfied for $U=\Omega(P)$
and $p = g_\Omega$.
Choose $R>0$ sufficiently large such that when $\lvert z \rvert > R$,
then $\lvert(P(z))\rvert > 2 \lvert z \rvert$,
and let $V = \C \setminus \overline{\D}_R$.
Then $P^k$ has no roots in $\overline{V}$ for any $k = 1, 2, \ldots$, 
$V \subset P^{-1}(V) \subset P^{-2}(V) \subset \cdots$
and $\Omega(P) = \cup_{k > 0} P^{-k}(V)$.
Let $L \subset \Omega(P)$ be an arbitrary compact set.
By compactness there exists $k_0$, 
such that $L \subset P^{-k}(V)$ when $k \geq k_0$.
In particular $P^k$ has no zeros in $L$ when $k \geq k_0$,
showing that condition~\itemref{uniformboundroots}\ is satisfied.

It is well known that $g_\Omega\circ P = d g_\Omega$ on $\Omega(P)$,
and it follows that
\[
  g_\Omega
  = \lim_{k\to\infty}d^{-k}\log|P^k|
  = \lim_{k\to\infty}d^{-k}\log|q_k|
\]
locally uniformly on $\Omega(P)$,
see for example \cite[Cor. 6.5.4]{Ransford}.
The Green's function $g_\Omega$ is harmonic and non-constant on $\Omega(P)$,
so condition~\itemref{balayage}\ is fulfilled with $p=g_\Omega$,
$N_2=N_1=\N$, $d_k=0$ for all $k$, and $\Exc=\emptyset$.

Hence we can apply the main theorem and obtain
\[
    \xi_{k,m} - \xi_k 
    \underset{k\to\infty}{\longrightarrow} m \xi_{p'},
    \text{ for } m = 0, 1, 2, \ldots 
\]
We have already seen that given a compact set $L\subset \Omega(P)$,
$P^k$ has no zeros in $L$, for $k$ sufficiently big,
and it follows that $\xi_k \to 0$.
Hence  
\[
  \xi_{k,m}
  \underset{k\to\infty}{\longrightarrow} m \xi_{p'}, \text{ for } m = 0, 1, 2, \ldots 
\]

Let $C \subset \Omega(P)$ be an arbitrary compact set
having no critical points of $g_\Omega$ in its boundary.
Let $s$ denote the number of critical points of $g_\Omega$
in $C$ and $t_k$ denote the number of zeros in $C$ 
of the $m$th derivative of $P^k$ both counted with multiplicity.
We must show that for every $m$, there exists $k_0$,
such that when $k \geq k_0$ then $t_k = sm$.

Let $f : \Omega(P) \to [0, 1]$, $f \in \cont_c(\Omega(P))$ satisfy
$f(z) = 0$ outside $C$ and $f(z) = 1$ at every critical point
of $g_\Omega$ in $C$.
Then
\[
  t_k \geq \xi_{k, m}(f)
  = m\xi_{g'_\Omega}(f) + o(1)
  = sm + o(1) \text{ as }k \to \infty.
\]
Since $t_k$ and $sm$ are integers, it follows that $t_k \geq sm$
for $k$ sufficiently big.

To see the other inequality,
let $F : \Omega(P) \to [0, 1]$, $F \in \cont_c(\Omega(P))$,
satisfy $F(z) = 1$ when $z \in C$ and $F(z) = 0$
for every critical point of $g_\Omega$ outside $C$.
Such a function exists since $g_\Omega$ is non-constant and harmonic
in $\Omega(P)$, so the critical points cannot accumulate on $C$.
Then
\[
  t_k \leq \xi_{k, m}(F)
  = m\xi_{g'_\Omega}(F) + o(1)
  = sm + o(1) \text{ as }k \to \infty.
\]
Hence $t_k \leq sm$ for $k$ sufficiently big, finishing the proof.
\end{proof}

We have illustrated the case $P(z) = z^2 + \frac{1}{2}$
in Figure~\ref{fig:dyn}.

\begin{figure}
   \includegraphics[width=0.48\textwidth]{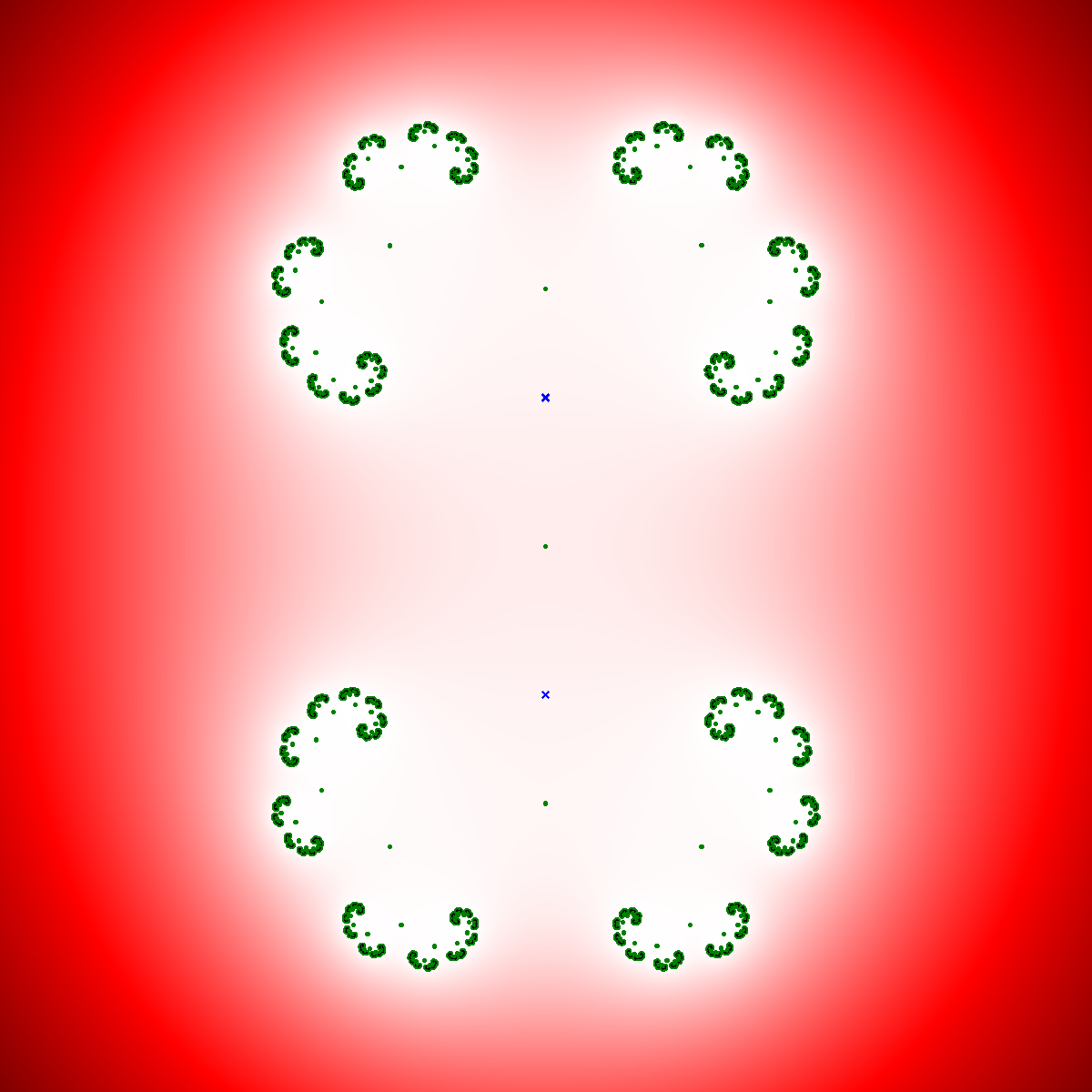} 
   \hfill \includegraphics[width=0.48\textwidth]{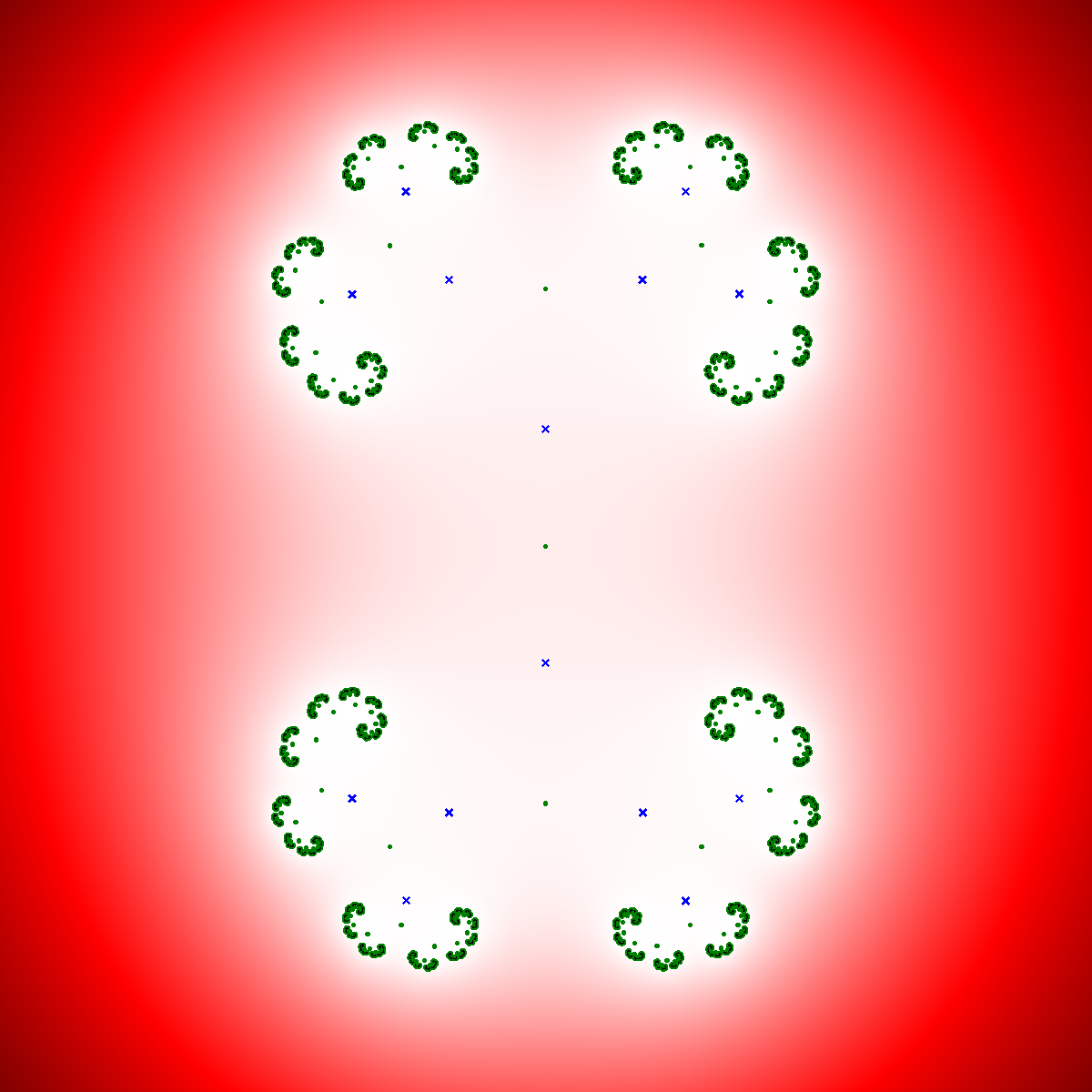} \\
   \includegraphics[width=0.48\textwidth]{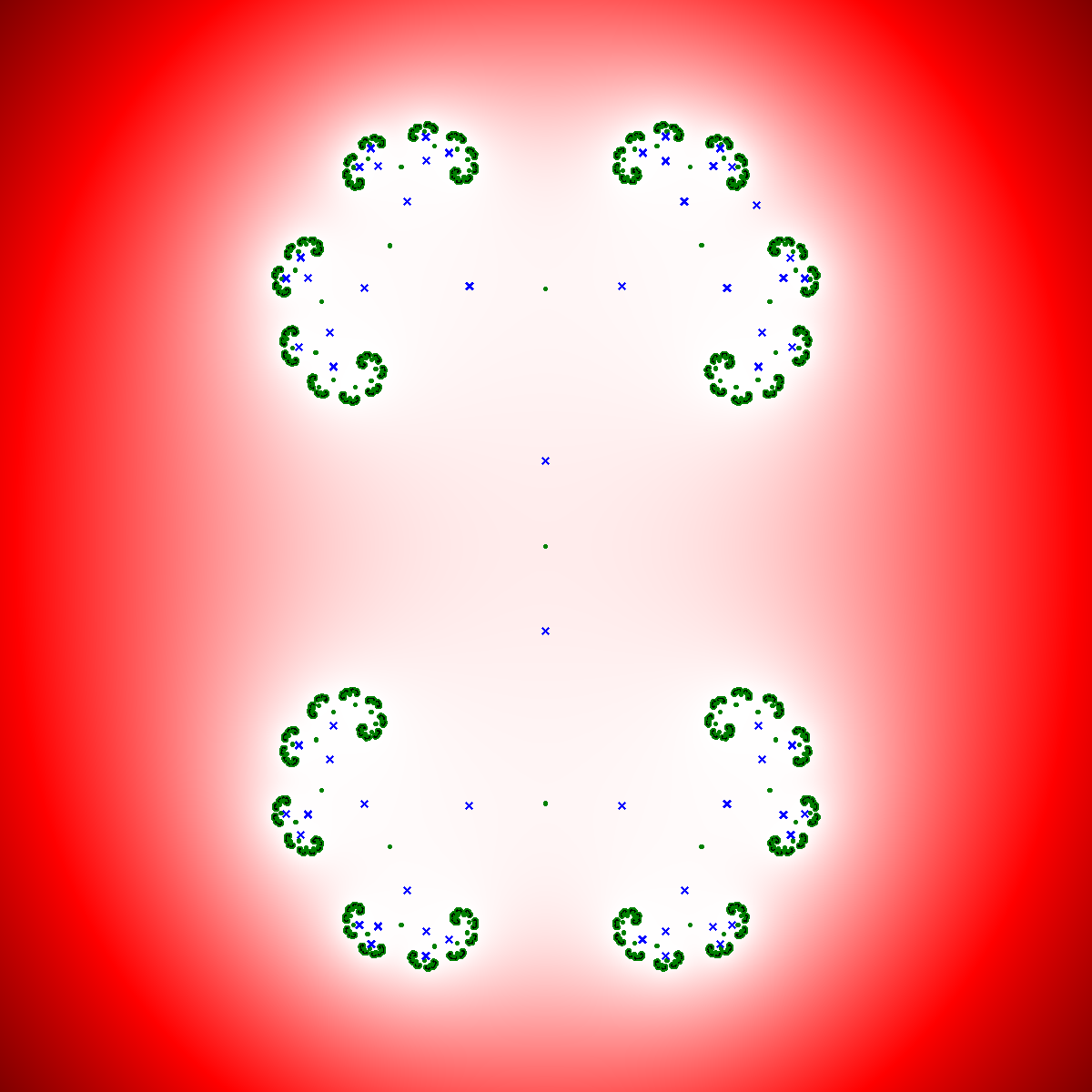} 
   \hfill
   \includegraphics[width=0.48\textwidth]{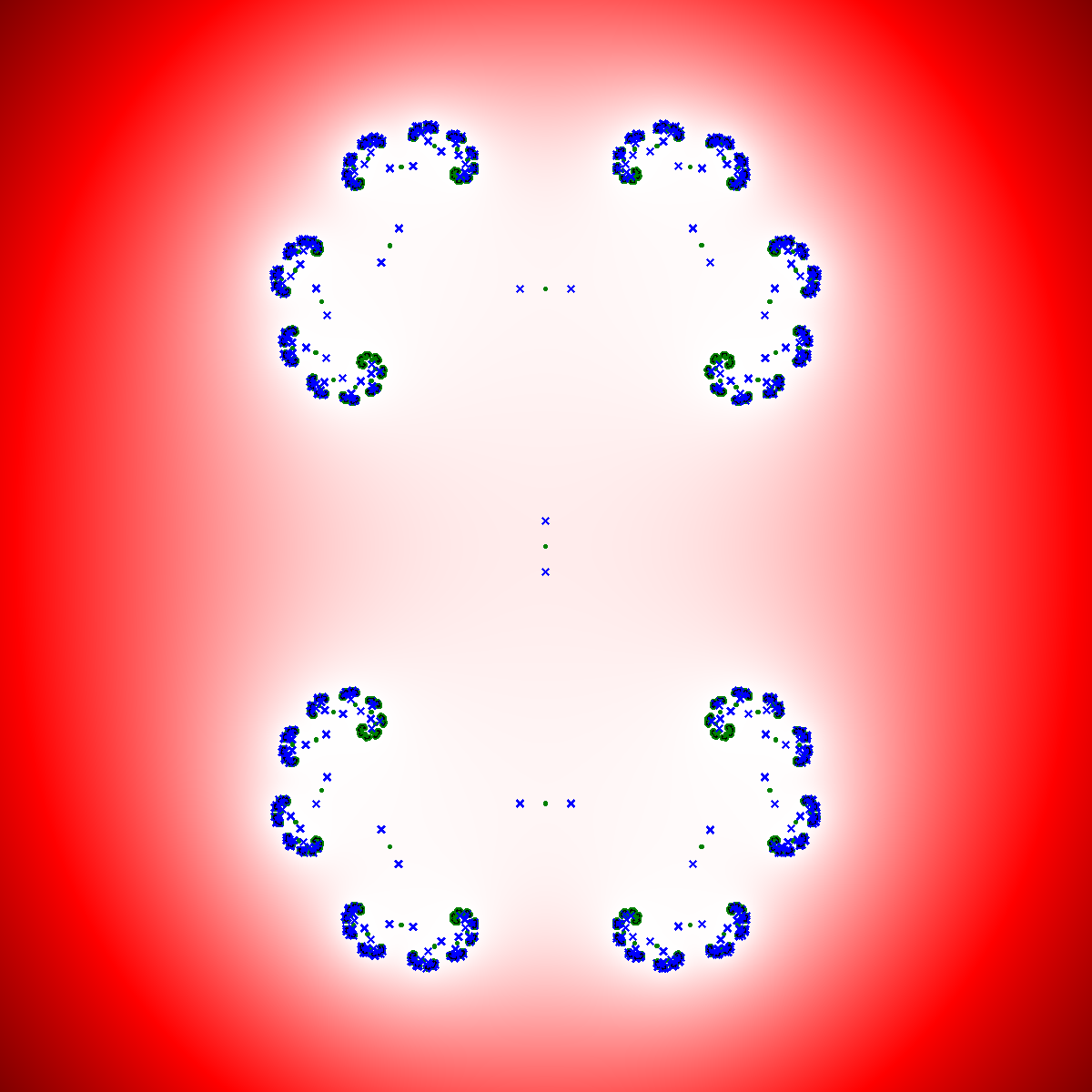} 
\caption{We illustrate the dynamics of $P(z)=z^2+\frac{1}{2}$
together with the roots of the second derivative of $P^k$,
for $k = 2, 4$ (first row) and $k = 6, 10$ (second row).
Every picture corresponds to the region
$\{x +iy : -3/2 < x < 3/2, -3/2 < y < 3/2\}$.
The filled-in Julia set $K(P)$ is shown in black.
It is well-known that it is a Cantor set for this particular polynomial.
The critical points of the Green's function $g_\Omega$ is shown
in green, whereas the value of $g_\Omega$ is suggested by shades
of red.
Finally, the roots of the second derivative of $P^k$ are marked with
blue crosses.
A consequence of Proposition~\ref{pro:dyn}
is that
there will be two roots of $\frac{dP^k}{dz^2}$ converging
to each critical point of $g_\Omega$ as $k\to \infty$.
Looking at the critical point at the origin, it does seem
that two critical points of $g_\Omega$, lying on the imaginary
axis, get closer and closer to $0$ as $k$ is increased.
} \label{fig:dyn}
\end{figure}

\paragraph{Families of orthogonal polynomials.}
Let $\mu$ be a Borel probability measure whose
support $K \subset \C$ is a compact and infinite set,
and consider the family $(q_k)_k$ of orthogonal
polynomials associated to $\mu$.
Let $U$ be the unbounded connected component of
$\C\setminus K$.
Then $U$ is root-sparse for the sequence $(q_k)_{k}$.
This was originally shown
by Fejér \cite{Fejer}; see also
\cite[Lemma 1.1.3]{StahlTotik}.
In particular, any limit point $\nu$ of the pre-compact
family $\mu_k$ has support in $K$.
This leads to the following proposition.
\begin{prop}\label{pro:ortho}
Suppose that a subsequence $\mu_{k_j}$
of root distributions converges to $\nu$.
Then
  \[
    \xi_{k_j,m} - \xi_{k_j} 
    \underset{k\to\infty}{\longrightarrow}\ 
    m \xi_{p_\nu'} 
    \text{ on } \Omega\text{, for } m = 0, 1, 2, \ldots
  \]
\end{prop}

\begin{proof}
    The proposition follows immediately from the
    preceeding remarks and
    \thmref{main}.
\end{proof}

\paragraph{Extremal families in the sense of Widom.}
Our next application deals with a quite general class of extremal
families of polynomials introduced by Widom, see \cite{Widom}. 

Let $K\subset\C$ be a compact, non-polar set. A measure $\mu$ on $K$
is called \emph{admissible} (in the sense of Widom) if there exists a
family of closed subsets  $K_t\subset K, 0<t<t_0$ which satisfies
\begin{enumerate}
       \item $\lim_{t\to 0}c(K_t)=c(K)$
    \item $\forall A\subset K$ satisfying $\mu(A)=0$:
    $\liminf_{t\to 0}\omega_t(A\cap K_t)=0$,
\end{enumerate}
where $\omega_t$ denotes the equilibrium measure on $K_t$ and $c(K_t)$, 
$c(K)$ denote the logarithmic capacities of the respective sets. 
Widom remarks that in the simplest case, where $K_t=K$ for all $t$, condition~2. is equivalent to $\omega_K \ll \mu$, i.e. $\omega_K$ is absolutely continuous with respect to $\mu$. 

Let $\mu$ be a measure on $K$, with infinite support. Suppose $\Phi:\R_+\cup {0}\to \R_+\cup {0}$ is a non-negative, continuous and increasing function, such that 
\[
s_n=o(t_n)\Rightarrow \Phi(s_n)=o(\Phi(t_n)).
\] 
For example $\Phi_p(t) = t^p$, where $p>0$.

Given $\mu$ and $\Phi$, Widom considers for $k\in\N$ 
the minimum problem 
$$m_k=\inf \int \Phi(|z^k+a_1z^{k-1}+ \ldots + a_k|)d\mu(z), $$
where the infimum is taken over all $a_1, \ldots ,a_k\in \C$. 
He shows that for each $k$ the infimum is a minimum 
realized by some (monic) polynomial $q_k$ of degree $k$.
We call such a family of polynomials $(q_k)_k$ a Widom-extremal family. 
Examples include the extremal families in $L^p(\mu)$, $p>0$ 
such as the sequence of polynomials orthogonal with respect to $\mu$, 
(i.e.~$p=2$).

The following result is a corollary of \thmref{maintechnical} and results of Widom \cite{Widom}.
\begin{prop}%[Corollary of \thmref{maintechnical}]
Let $K\subset \C$ be a compact set, let $\Omega$ denote the unbounded connected component of $\C\setminus K$ and $g_\Omega$ the corresponding Green's function with pole at $\infty$. 
Suppose $\mu$ is an admissible measure on $K$, 
and $(q_k)_k$ is a Widom-extremal family.
Then 
  \[
    \xi_{k,m} - \xi_k 
    \underset{k\to\infty}{\longrightarrow}\ m \xi_{p_\mu'}, 
    \text{ on } \Omega\text{ for } m = 0, 1, 2, \ldots
  \]
\end{prop}
\begin{proof}
 Widom shows in \cite[Lemma 4]{Widom} that Condition \itemref{uniformboundroots}\ of Theorem~\ref{maintechnical} is satisfied. 
 
 Moreover, in \cite[Corollary p. 1007]{Widom} he shows that $\lim_{k\to\infty}\frac{1}{k}\log |q_k(z)|=g_\Omega(z)+\log c(K) =
 p_\omega(z)$ locally uniformly on the complement of the convex hull of $K$. 

It follows that for any infinite set $N_1\subset\N$ there exist an infinite subset $N_2\subset N_1$ and a subset $\Exc\subset \Omega$
without accumulation points in $\Omega$ such that 
\[
    \frac{1}{k}\log |q_k(z)|\to g_\Omega+\log c(K),\quad
    \textrm{locally~unifomly in $\Omega\setminus\Exc$ as $k\to\infty$ in $N_2$},
\]
i.e. that Condition \itemref{balayage}\ is satisfied with $p=g_\Omega+c(K)$ and $d_k=0$.
The proof is left to the reader. 
For a similar proof see the proof of \cite[Proposition 2]{HPU3}.
\end{proof}

% *********************************************
%
\section{Proofs of theorems}\label{sec:proofs}
%
% *********************************************

We prove Theorem~\ref{maintechnical} and derive Theorem~\ref{main}. 
We start with a lemma. 

\begin{lemma}\label{BalayagePotConv}
Assume condition \itemref{balayage}\ in Theorem~\ref{maintechnical} holds and
$W \subset U$ is an open and bounded set
such that $q_k$ is non-vanishing 
on $W$ for $k \in N_2$ sufficiently big.
  Then locally uniformly on $W$
    \begin{enumerate}
    \item $d_k + \frac{1}{n_k}\log\left\lvert q_k\right\rvert \to p$
    as $k \to \infty$ in $N_2$ and
    \item $\frac{q_k'}{n_k q_k} \to p'$ as $k \to \infty$ in $N_2$.
  \end{enumerate}
  
  Furthermore, if $K\subset W$ is a compact set containing no zeros of $p'$, then $q_k'$ has no zero in $K$ for $k$ sufficiently big. 
\end{lemma}
Notice that part 1.~of the lemma implies that 
we can always assume that the exceptional set $\Exc$ 
in condition \itemref{balayage}\ only
consists of accumulation points of zeros of the family $(q_k)_k$, 
$k\in N_2$.

\begin{proof}
With $N_1\supset N_2$ and $\Exc\subset U$ as in
Condition~\itemref{balayage}\  
it is enough to prove the first statement, then the second follows 
 by differentiation (see \cite[Lemma 1]{HPU3} for details). 
Moreover it is enough to prove that if $z_0\in\Exc\cap W$, 
then $d_k + \frac{1}{n_k}\log\left\lvert q_k\right\rvert 
= d_k+p_{\mu_k} \to p$  
uniformly on a compact neighbourhood of $z_0$ in $W$ as $k \to \infty$ in $N_2$. 

So suppose $z_0\in\Exc\cap W$ and choose $r>0$ 
such that $\overline{\D}(z_0, r)\subset W$ and 
$\Exc\cap\overline{\D}(z_0, r) = \{z_0\}$. 
This is possible since $W\subset U$ is open 
and $\Exc$ is without accumulation points in $U$. 
Then by condition \itemref{balayage}\  
the harmonic functions $d_k+p_{\mu_k} - p$ on $W$ 
converge uniformly to the constant function $0$ 
on $\{z : \lvert z - z_0\rvert = r\}$. 
Whence by the maximum principle for harmonic functions 
the convergence is uniform on the closed disk $\overline{\D}(z_0, r)$.

Finally, if $K\subset W$ is a compact set containing no zeros of $p'$, then 
\[\frac{q_k'(z)}{n_k q_k(z)}= p'(z)+ o(1)=p'(z)(1+o(1))\]
uniformly on $K$ as $k \to \infty$ in $N_2$, so that 
\[q_k'(z)=n_k q_k(z)p'(z)(1+ o(1)),\]
whence $q_k'$ has no zero in $K$ for $k$ sufficiently big. 
\end{proof}

We prove Theorem~\ref{maintechnical} in two steps. 
First we prove the conclusion of the Theorem for $m=1$ in Proposition~\ref{ResultFirsderivatives}. 
Secondly we prove in Proposition~\ref{Hereditary} 
that the conditions \itemref{uniformboundroots}\ and
\itemref{balayage}\ are \emph{hereditary}, 
i.e. if the conditions are satisfied for a sequence of polynomials
$(q_k)_k$, then they are satisfied for the sequence of derivatives $(q'_k)_k$.
%
%  PROPOSITION
%
\begin{prop}\label{ResultFirsderivatives}
Theorem~\ref{maintechnical} holds when $m=1$. 
That is, under the assumptions of Theorem~\ref{maintechnical}
we have convergence of divisors on $U$
  \[
    \xi_{k,1} - \xi_k 
    \to \xi_{p'}\qquad\textrm{as}\qquad k\to\infty. 
  \]
\end{prop}
\begin{proof}
Let $\phi\in\cont_c(U)$ be arbitrary. We must show that 
\begin{equation}\label{equationconvergenceWeil}
    \xi_{k,1}(\phi) - \xi_k(\phi) 
    - \xi_{p'}(\phi)\to 0,\qquad\textrm{as }k\to\infty. 
\end{equation}
It is enough to show that for any infinite subset $N_1\subset\N$, 
there is a further infinite subset $N_2\subset N_1$ such that 
\eqref{equationconvergenceWeil} holds for $k \to \infty$ in $N_2$. 
So let $\epsilon>0$ be arbitrary and let $N_1\subset\N$ be an arbitrary infinite subset.

Let $K=\Supp(\phi)$ and define a 
%signed distance to $K$ 
function 
$v$ on $\C$ by 
\[
v(z)=\begin{cases}
-d(z, \partial K) \text{ for } z\in K\\
\phantom{-}d(z, \partial K) \text{ for } z\notin K,
\end{cases}
\]
where $d(\cdot,\cdot)$ denotes the euclidean distance. The function $v$ thus gives the signed distance to $\partial K$. For $\delta>0$ let $K_\delta := \{z \in\C: v(z)\leq \delta\}$, 
$2\delta_0 := d(K,\partial U) = \inf\{v(z) : z\in \C\setminus U\}$. 
%and set $K_0 = K_{\delta_0}$. 
Let $M = M(K_{\delta_0})$ be the uniform upper bound 
on the number of zeros in $K_{\delta_0}$ given by condition~\itemref{uniformboundroots}

Let $(z_{k,j})$, $j=1, \ldots, n_k$ denote the roots of $q_k$ repeated with respect to multiplicity. We can arrange the zeros $(z_{k,j})_{j=1}^{n_k}$ of $q_k$, 
so that 
$v(z_{k,j}) \leq v(z_{k,j+1})$ for all $1\leq j< n_k$. 
Then starting from $N_1$ we can find a 
subset $N_2\subset N_1$ such that the first $M+1$ roots converge as $k\to\infty$ in $N_2$. We let $z_j=\lim_{\stackrel{k\to\infty}{k\in N_2}} z_{k,j}$. 
Notice that $v(z_j)\leq v(z_{j+1})$ for $1\leq j\leq M$ and $v(z_{M+1})\geq \delta_0$, so that $z_{M+1}\notin K$. 
If $z_1\notin K$ set $m=0$. 
Otherwise let $m$, $1\leq m\leq M$, be maximal such that 
$z_j\in K$ for $j\leq m$. 
Let $E := \{ z_j : 1\leq j \leq m\}$, then by Lemma~\ref{BalayagePotConv} we can assume that $\Exc\cap K = E$. 
Let $w_1, w_2, \ldots$ be a labeling of the critical points of $p$ repeated with multiplicity and with $v(w_j)\leq v(w_{j+1})$. 
If $w_1\notin K$ we let $l=0$ and otherwise we let $l$ be maximal 
with $w_l\in K$.
Let $F := \{w_j : 1\leq j \leq l\}$ and $N=m+l$. 

Since $\phi$ is continuous with compact support, it is uniformly continuous. 
Let $\delta_1>0$ be such that 
\begin{equation}\label{eq:uniformcont}
    |\phi(x)-\phi(y)|<\epsilon/N, 
\text{ when } |x-y|\leq \delta_1.
\end{equation}
Define 
$$
\delta_2 = \frac{1}{4}\min( \{|x-y| : x\not= y, x, y\in E\cup F\}\cup
\{\delta_0, v(z_{m+1}), \delta_1, v(w_{l+1})\}),
$$
where we define $v(w_{l+1}) = \infty$, 
if $F$ contains all critical points of $p$.

Let $k_0$ be such that $|z_{k,j}-z_j|< \delta_2$ for all $k\geq k_0$, $k\in N_2$, and all $1\leq j \leq m+1$.
%and set $K_2 = K_{2\delta_2}$, $K_3 = K_{3\delta_2}$. 
We claim that for $k \geq k_0$ and $k\in N_2$ 
\[
q_k^{-1}(0)\cap K \subset  
q_k^{-1}(0)\cap K_{3\delta_2}=\{z_{k,1}, \ldots z_{k,m}\}.
\]
If $j>m$ then $v(z_{k,j})\geq v(z_{k,m+1})>v(z_{m+1})-\delta_2\geq 3\delta_2$, 
whence $q_k^{-1}(0)\cap K_{3\delta_2}\subset\{z_{k,1}, \ldots z_{k,m}\}$.
Furthermore since $v(z_{k,m})< v(z_{m})+\delta_2\leq \delta_2$ the claim follows.

Let $D_z = \D(z, 2\delta_2)$.\\
\noindent\textbf{Claim 1.} For every $z\in E\cup F$ the number of zeros 
of $q_k'$ in $D_z$ equals the number of zeros of $q_k p'$ in $D_z$, 
all zeros counted with multiplicity.

\noindent\textbf{Proof of Claim 1.}  Define nested neighborhoods 
$D_1\subset D_2$ of $E\cup F$ by
$$
D_1 = \bigcup_{z\in E\cup F}\D(z,\delta_2) 
\qquad\text{and}\qquad
D_2 = \bigcup_{z\in E\cup F}D_z.
$$
Moreover define the compact set 
$$
L_2 :=
K_{2\delta_2}\setminus D_2
%\bigcup_{z\in E\cup F}\D(z,2\delta_2) 
$$
and notice that $L_2$ contains
$$
\partial D_2 = \bigcup_{z\in E\cup F}\partial D_z.
$$
Define an open neighborhood $W$ of $L_2$ by
$$
W :=\overset{\circ}{K}_{3\delta_2} \setminus 
\overline{D}_1.
$$
By construction $p'$ and each $q_k$, for $k \geq k_0$ and $k\in N_2$, 
do not vanish on $\overline{W}$. 
Combining this with Lemma~\ref{BalayagePotConv} 
shows that $q_k'$ has no zero in $\overline{W}$ 
for $k\in N_2$, $k\geq k_0$, increasing $k_0$ if necessary. 

Let 
\[
  \rho := \min\{|p'(x)| : x\in \partial D_2\},
\]
then, further increasing $k_0$ if necessary, we can assume by 2.~in Lemma~\ref{BalayagePotConv} that 
\[
\rho > \sup\left\{\left|\frac{q_k'(x)}{n_kq_k(x)} - p'(x)\right| : 
 x\in \partial D_2\right\},
\]
for $k\geq k_0$, $k\in N_2$.
The claim is now an immediate consequence of Rouché's theorem applied to each $D_z$. 

\noindent\textbf{Claim 2.} For every $z\in E\cup F$ 
$\lvert \phi(x) - \phi(y) \rvert < \epsilon/N$ for every $x, y\in D_z$

\noindent\textbf{Proof of Claim 2.} 
The claim follows from \eqref{eq:uniformcont} and the bound 
$diam(D_z) = 4 \delta_2 \leq \delta_1$.

Hence, when $k \geq k_0$, $k\in N_2$
\begin{align*}
\lvert\xi_{q_k'}(\phi)-\xi_{q_k}(\phi)-\xi_{p'}(\phi)\rvert 
&=
\left|\sum_{z\in E\cup F}
\left(\sum_{\substack{x\in D_z\\q_k'(x)=0}}\phi(x)\quad - 
\sum_{\substack{y\in D_z\\ (q_kp')(y)=0}}\phi(y)\right)\right|\\
&\leq \sum_{z\in E\cup F}
\left|\sum_{\substack{x\in D_z\\q_k'(x)=0}}\phi(x)\quad - 
\sum_{\substack{y\in D_z\\ (q_kp')(y)=0}}\phi(y)\right |\\
&\leq  
\sum_{z\in E\cup F}
\sum_{\substack{x\in D_z\\ q_k'(x)=0}} \epsilon/N = \epsilon
\end{align*}
where the roots in the sums are repeated according to multiplicity 
and the last inequality follows from Claims 1.~and 2.~above.
\end{proof}
In order to complete the proof of Theorem~\ref{maintechnical}, 
we must show that the union of the conditions 
\ref{uniformboundroots}.~and~\ref{balayage}.~are hereditary, 
i.e. are passed on to the sequence of normalized derivatives
$(\frac{1}{n_k}q_k')_k$. 
This is the content of the following proposition.
%
% PROPOSITION Heriditary
%
\begin{prop}\label{Hereditary}
  If $(q_k)_k$ is a sequence of polynomials as above, satisfying the conditions \itemref{uniformboundroots}\ and \itemref{balayage}\ of Theorem~\ref{maintechnical}, then the sequence of 
  normalized derivatives $(\frac{1}{n_k}q_k')_k$ 
  also satisfies the conditions 1.~and 2.~with the same harmonic function $p$.
\end{prop}
\begin{proof}
Assume that $(q_k)_k$ is a sequence of polynomials satisfying the conditions \itemref{uniformboundroots}\ and \itemref{balayage}\ of Theorem~\ref{maintechnical}. To alleviate notation we set $\nu_k:=\mu_{q_k'}$.

Inheritance of condition \itemref{uniformboundroots}~is a consequence of Proposition~\ref{ResultFirsderivatives} as follows.
Consider a compact set $K \subset U$, and let $f\in \cont_c(U)$ be a function $f:U \rightarrow [0, 1]$, which is $1$ on $K$. Then 
\begin{align}
  \nonumber\nu_k^\#(K) &= (n_{k}-1)\nu_k(K) \\
  & \leq \xi_{k,1}(f)
    = \xi_k(f) + \xi_{p'}(f)+ o(1)\label{eq:nu_k_estimate}  \\
  \nonumber & \leq \mu_k^\#(\Supp(f))+ s + o(1),
\end{align}
as $k\to\infty$, where the equality in \eqref{eq:nu_k_estimate} 
follows from Proposition~\ref{ResultFirsderivatives}, 
and $s$ is the number of critical points of $p_{\mu}$ in $\Supp(f)$. 
Condition \itemref{uniformboundroots}\ for the sequence of normalized derivatives $(\frac{1}{n_k}q_k')_k$ then follows from condition \itemref{uniformboundroots}\ for $(q_k)_k$.

To show inheritance of condition \itemref{balayage}\ we have to show that given any infinite set
$N_1\subset\N$ there exists an infinite subset $N_2\subset N_1$, a subset $\Exc'\subset U$ and a sequence $(d'_k)_k$ such that
$d'_k+p_{\nu_k} \to p$ locally uniformly on $U\setminus \Exc'$ for 
$k\to\infty$, $k\in N_2$. 

Let $v : U \to (0, \infty)$ denote the function $v(z) = |z| + 1/d(z,\partial U)$. 
Then $v$ is bounded on any compact subset of $U$ and for any $c>0$ 
the set $v^{-1}([0,c])$ is compact.

Assume we have an infinite set $N_1\subset\N$.
Let $z_{k,j}$ denote the roots of $q_k$ ordered 
such that
\begin{itemize}
    \item if $z_{k,j} \in U$ and $z_{k,l} \notin U$
    then $j < l$
    \item if $z_{k,j}, z_{k,l} \in U$
    then 
    $(j \leq l 
    \Leftrightarrow v(z_{k,j}) \leq v(z_{k,l}))$.
\end{itemize}
By a standard diagonal argument, there exists an infinite set
$N_2 \subset N_1$ and a sequence $(z_j) \subset \CC$,
such that $z_{k,j} \to z_j$ as $k \to \infty$ in $N_2$.

Let $E = U \cap \{z_j: k=1,2,\ldots\}$.
The set $E$ has no accumulation points
in $U$ by root sparsity, condition \itemref{uniformboundroots}\ 

We claim that if a compact set $L \subset U$ does not meet $E$,
then there exists $k_0 \in N_2$, such that when $k \geq k_0$, $k \in N_2$,
$q_k$ has no zeros in $L$.

To see this, let $\hat{v} = \max\{v(z): z \in L\}$ and $K = \{z\in U: v(z) \leq \hat{v}\}$.
Then $L \subset K \subset U$, and $K$ is compact.
By root-sparsity, only a finite number $j_0$ of limit points $z_j$ are elements of $K$.
The points are ordered such that these points must be $z_1, \ldots, z_{j_0}$.
Choose $\epsilon > 0$ such that $d(z_j, L) \geq \epsilon$ for $j=1, \ldots, j_0$
and such that $d(z_{j_0+1}, K) \geq \epsilon$.

By convergence of the roots sequences $(z_{k,j})_k$,
there exists $k_0 \in N_2$, such that when $k \geq k_0$ 
and $k \in N_2$,
we have $d(z_{k,j}, z_j) < \epsilon$ for $j = 1, 2, \ldots, j_0 + 1$.

Let $k \geq k_0$ and $k \in N_2$ be arbitrary.
We show that $z_{k,j} \notin L$, by dividing into the three cases $j < j_0 + 1$, $j=j_0 + 1$
and $j > j_0 + 1$.

When $j < j_0 + 1$, we have $d(z_{k, j}, L) \geq d(z_j, L) - d(z_j, z_{k,j}) > 0$.

When $j = j_0 + 1$ we have $d(z_{k, j}, K) \geq d(z_j, K) - d(z_j, z_{k,j}) > 0$.
So $z_{k,j_0+1} \notin K \supset L$.
It follows that $v(z_{k,j_0+1}) > \hat{v}$ which we will use in the third and last case.

When $j > j_0 + 1$, either $z_{k,j} \notin U \supset L$,
or $z_{k,j} \in U$ and then by the ordering $v(z_{k,j}) \geq v(z_{k, j_0+1}) > \hat{v}$,
so $z_{k,j} \notin K \supset L$.
This proves the claim.

It then follows from condition \itemref{balayage}\ for $(q_k)_k$ and Lemma~\ref{BalayagePotConv} that we can assume $E = \Exc$ and that 
\[\frac{1}{n_k}\log\left\lvert q_k %\frac{q_k}{\gamma_k}
\right\rvert + d_k \to p
\qquad\textrm{as } k\to\infty \textrm{ in } N_2,\]
locally uniformly on $U\setminus \Exc$.

Let $F=\{z\in \C : p'(z)=0\}$. 
Also from Lemma~\ref{BalayagePotConv} 
we have $\frac{q_k'}{n_k q_k} \to p'$ 
locally uniformly on $U\setminus \Exc$ as $k\to\infty$ in $N_2$, 
so that 
\[\frac{q_k'}{n_k q_k} = p'(1+o(1))
\qquad\textrm{as } k\to\infty \textrm{ in } N_2,\]
locally uniformly on $U\setminus (\Exc\cup F)$. 
Taking $\log$ and dividing by $n_k-1$ we obtain 
\[\frac{1}{n_k-1}\log\left\lvert\frac{q'_k}{n_k}\right\rvert - \frac{1}{n_k-1}\log|q_k| = o(1)
\qquad\textrm{as } k\to\infty \textrm{ in } N_2.\]
Finally, since $p$ is uniformly bounded on compact 
subsets of $U$ we obtain 
\begin{align*}
p_{\nu_k}=\frac{1}{n_k-1}\log\left\lvert \frac{q'_k}{n_k}\right\rvert
&= \left(1+\frac{1}{n_k-1}\right)\frac{1}{n_k}\log\left\lvert q_k \right\rvert + o(1)\\
&=\left(1+\frac{1}{n_k-1}\right)(p-d_k + o(1)) =  p-d_k'+o(1),
\end{align*}
locally uniformly on $U\setminus \Exc\cup F$ as $k\to\infty$ in $N_2$. 
Thus, condition \itemref{balayage}\ holds for the sequence of normalized derivatives,
with $N_2$ as described above, $\Exc'=\Exc\cup F$ and $d_k'=d_k(1+\frac{1}{n_k-1})$.
\end{proof}

It is straightforward to prove Theorem~\ref{maintechnical} from 
Propositions~\ref{ResultFirsderivatives} and
\ref{Hereditary}.
\begin{proof}
    We must show
    \begin{equation}\label{eq:divisorrel}
      \xi_{k,m}-\xi_k \to m\xi_{p'}
      \text{ as }k\to \infty
    \end{equation}
    for $m=0, 1, 2, \ldots$
    
    Applying Proposition~\ref{Hereditary} inductively, 
    each of the sequences $\left(\frac{(n_k-m)!}{n_k!}q_k^{(m)}\right)_k$ 
    satisfies the requirements of Theorem~\ref{maintechnical}, 
    and by Proposition~\ref{ResultFirsderivatives}
    \[ 
      \xi_{k,m+1}-\xi_{k,m} \to \xi_{p'}
      \text{ as }k\to \infty
    \]
    for $m=1,2,\ldots$
    Equation \eqref{eq:divisorrel} now follows by summation.
    Indeed 
    \begin{equation*}
        \xi_{k,m} - \xi_k
        = \sum_{l=1}^m (\xi_{k,l} - \xi_{k,l-1}) 
         \underset{k\to\infty}{\longrightarrow} m\xi_{p'} 
    \end{equation*}
    where $\xi_{k,0} := \xi_k$.
\end{proof}

Having proven Theorem~\ref{maintechnical}, we now prove
Theorem~\ref{main}.
\begin{proof}
First assume that $U$ is connected.
We will show that the hypotheses of \thmref{main}
imply the hypotheses of \thmref{maintechnical}, with
$p = p_\mu$.
This will show the theorem in the case where $U$ is connected.

Let $N_1 \subset \N$ be an infinite subset.

We define a function $v$ and order the roots in
the same way we did in the proof of Proposition~\ref{Hereditary}.
That is, we define the function $v : U \to (0, \infty)$ by $v(z) = \absv{z} + 1/d(z, \partial U)$.
For each $k$, order the roots $z_{k,j}$, $j=1,\ldots,n_k$ of $q_k$,
such that the roots inside $U$ are ordered first and in ascending order
with respect to $v$.
More formally, label the roots such that
\begin{itemize}
\item if $z_{k,j} \in U$ and $z_{k,l} \notin U$
    then $j < l$, and
\item if $z_{k,j}, z_{k,l} \in U$
    then 
    $(j \leq l 
    \Leftrightarrow v(z_{k,j}) \leq v(z_{k,l}))$.
\end{itemize}

By a standard diagonal argument, there exists an infinite set $N_2 \subset N_1$,
such that for every $j$, $z_{k,j} \to z_j \in \CC$ as $k\to \infty$ in $N_2$.
Notice that the limits $z_j \in \CC$ fulfill the same two ordering properties as $z_{k,j}$.

Put $\Exc = U \cap \{z_j : j=1,2,\ldots\}$.
By root-sparsity, $\Exc$ has no accumulation points in $U$. 
We claim that if $L \subset U$ is a compact set, not meeting $\Exc$, then $q_k$ has no roots in $L$, for $k\in N_2$ sufficiently big. 
The claim follows as a consequence of the ordering of the roots and root sparsity (condition \itemref{uniformboundroots}), and was proven in the proof of Proposition~\ref{Hereditary}.

We have $p_\mu'(z) = \int \frac{d\mu(w)}{z-w}$ on $U$,
and $p_{\mu_k}'(z) = \int \frac{d\mu_k(w)}{z-w}$ when $q_k(z) \neq 0$.
From this and the claim, it follows that
$p_{\mu_k}' \to p_{\mu}'$ uniformly on compact subsets of $U \setminus \Exc$.

Fix some base-point $z_0 \in U$.
Since $U \setminus \Exc$ is connected, it follows from the integral representation
\eqref{eq:recuperatepot} that locally uniformly in $U \setminus \Exc$
\[ p_{\mu_k} - p_{\mu_k}(z_0) \to p_\mu - p_\mu(z_0)
\text{ as }k\to \infty \text{ in }N_2.
\]
Hence
\[ 
  p_{\mu_k} + d_k \to p_\mu
\]
locally uniformly in $U \setminus \Exc$ as $k \in N_2$ tends to infinity,
where $d_k = p_\mu(z_0) - p_{\mu_k}(z_0)$.

This shows that the hypotheses of \thmref{maintechnical} are satisfied
and \thmref{main} follows in the case where $U$ is connected.

It remains to consider the case where $U$ is disconnected.
We must show
\begin{equation*}
(\xi_{k,m}-\xi_k - m\xi_{p_{\mu}'})(\phi) \to 0 \text{ as }k \to \infty
\end{equation*}
for any $\phi \in \cont_c(U)$.
Let $K \subset U$ be the support of $\phi$.
Since the connected components of $U$ form an open cover of $K$,
compactness of $K$ implies that only finitely many of them
meet $K$. Denote those finitely many components
by $U_j$, $j=1,\ldots, N$,
and notice $K \subset \bigcup_{j=1}^N U_j$.

Let $\lambda_{U_j}$ be the characteristic function, i.e.\ $\lambda_{U_j}$
is $1$ on $U_j$ and zero elsewhere.
We have $\phi = \sum_{j=1}^N \lambda_{U_j} \phi$
and $\lambda_{U_j} \phi \in \cont_c(U_j)$.
Hence
\[
(\xi_{k,m}-\xi_k - m\xi_{p_{\mu}'})(\phi)
= \sum_{j=1}^N
(\xi_{k,m}-\xi_k - m\xi_{p_{\mu}'})(\lambda_{U_j} \phi)
\]
Since $\lambda_{U_j} \phi \in \cont_c(U_j)$, and we have just shown that
\thmref{main} holds on connected open sets such as $U_j$,
we have 
\[
(\xi_{k,m}-\xi_k - m\xi_{p_{\mu}'})(\lambda_{U_j} \phi) \to 0 \text{ as }k\to \infty,
\]
for each $j = 1, 2, \ldots, N$,
completing the proof.
\end{proof}

It only remains to prove Corollaries \ref{corollaryconvex} and \ref{corollaryPonvex}.

\begin{proof}[Proof of Corollary \ref{corollaryconvex}]
The result follows, if we can show that for every infinite set $N_1 \subset \N$,
there exists an infinite subset $N_2 \subset N_1$, such that
\[
      \xi_{k,m} - \xi_k \to 0 \text{ on }U, \text{ as $k \to \infty$ in $N_2$.}
\]

Let $N_1 \subset \N$ be an arbitrary infinite subset.
Since the root loci are uniformly bounded, $(\mu_k)_k$ is a precompact family.
Hence there exists an infinite subset $N_2 \subset N_1$ and a Borel probability measure $\mu$ with compact support,
such that $\mu_k \wto \mu$ as $k \to \infty$ in $N_2$.
Since $V$ contains the support of $\mu$ and is convex, $p_\mu$ has no critical points
in $U$ and in particular $\xi_{p_\mu'}=0$.
Convexity of $V$ implies that 
every connected component of $U$ is unbounded.
Therefore $p_\mu$ is not constant on any component of $U$. 
By \thmref{main}, $\xi_{k,m} - \xi_k \to 0$ on $U$ as $k \to \infty$ in $N_2$.
\end{proof}

\begin{proof}[Proof of Corollary \ref{corollaryPonvex}]
    The potential $p_\omega = g_U+I(\omega)$ is harmonic and not constant on $U$. It has no critical points because $K$ is connected.
\end{proof}

\paragraph{Acknowledgments}
The authors would like to thank the
Danish Council for Independent Research |  Natural Sciences for
support via the grant DFF--1026--00267B.
%
%
% Bibliography
%
%
\bibliographystyle{plain}
\bibliography{limitmeasure.bib}
\end{document}